\documentclass[11pt]{amsart}
\usepackage[a4paper,margin=1.05in]{geometry}
\usepackage{amsmath,amssymb,amsthm,mathtools}
\usepackage{enumitem}

\usepackage{hyperref}
\usepackage{mathrsfs}
\usepackage{microtype}
\hypersetup{colorlinks=true,linkcolor=blue,citecolor=blue,urlcolor=blue}

\makeatletter
\@namedef{subjclassname@2020}{\textup{2020} Mathematics Subject Classification}

\makeatother

\newtheorem{theorem}{Theorem}[section]
\newtheorem{proposition}[theorem]{Proposition}
\newtheorem{lemma}[theorem]{Lemma}

\theoremstyle{remark}
\newtheorem{remark}[theorem]{Remark}

\newtheorem*{theoremA}{Theorem A}
\newtheorem*{schurtest}{Schur's test}

\newcommand{\D}{\mathbb D}

\newcommand{\Hcal}{\mathcal H}
\newcommand{\B}{\mathrm B}

\newcommand{\norm}[1]{\left\lVert #1\right\rVert}

\title[Hilbert matrix norms]{Hilbert matrix norms on weighted Bergman spaces: even exponents and a counterexample to the beta formula}

\author{Hasi Wulan}
\address{Department of Mathematics, Shantou University, Shantou, Guangdong 515821, People's Republic of China}
\email{wulan@stu.edu.cn}

\author{Mengmeng Zhou}
\address{Department of Mathematics, Shantou University, Shantou, Guangdong 515821, People's Republic of China}
\email{25mmzhou@stu.edu.cn}

\author{Jian-Feng Zhu}
\address{Department of Mathematics, Shantou University, Shantou, Guangdong 515821, People's Republic of China}
\email{flandy@stu.edu.cn}

\date{\today}

\begin{document}

\begin{abstract}
Let $A^p_\alpha$ be the weighted Bergman space on the unit disk, where $\alpha>-1$.  For
\[
 f(z)=\sum_{k=0}^{\infty}a_kz^k\in A^p_\alpha,
\]
consider the Hilbert matrix operator
\[
 \Hcal f(z)=\sum_{n=0}^{\infty}
 \left(\sum_{k=0}^{\infty}\frac{a_k}{n+k+1}\right)z^n
 =\int_0^1\frac{f(t)}{1-tz}\,\mathrm dt.
\]
For even exponents $p=2m$, we prove
\[
 \norm{\Hcal}_{A^{2m}_\alpha\to A^{2m}_\alpha}
 =\B(a,1-a),\qquad a=\frac{\alpha+2}{2m},
\]
whenever $0<a\le m/(2m-1)$.  The same formula holds throughout the admissible range for $p=2,4,6,8,10$.

The beta formula is not valid for all admissible parameters.  Put
\[
 a_0=\frac{800001}{1000000},\qquad \alpha_p=a_0p-2.
\]
For every real $p\ge1100000$,
\[
 \norm{\Hcal}_{A^p_{\alpha_p}\to A^p_{\alpha_p}}
 >\B(a_0,1-a_0).
\]
The counterexample uses the fixed function
\[
 f_0(z)=(1-z^2)^{-4/5}
 =\sum_{k=0}^{\infty}\frac{(4/5)_k}{k!}z^{2k}.
\]
A rigorous interval estimate at $p=1100000$, followed by monotonicity in $p$, gives the whole half-line.  In particular, failure occurs for every even $p=2m$ with $m\ge550000$.

\end{abstract}


\maketitle

\section{Introduction and main results}

Let $\mathrm{d}A(z)=\pi^{-1}\,\mathrm{d}x\,\mathrm{d}y$ be normalized area measure on the unit disk $\D$.  For $\alpha>-1$ set
\[
  \mathrm{d}A_\alpha(z)=(\alpha+1)(1-|z|^2)^\alpha \mathrm{d}A(z),
\]
and let $A^p_\alpha$ be the analytic weighted Bergman space with norm
\[
  \norm{f}_{A^p_\alpha}^p=\int_\D |f(z)|^p\,\mathrm{d}A_\alpha(z).
\]
For a polynomial $f(z)=\sum_{k\ge0}a_kz^k$, the Hilbert matrix operator is
\[
  \Hcal f(z)
  =\sum_{n=0}^{\infty}\left(\sum_{k=0}^{\infty}\frac{a_k}{n+k+1}\right)z^n
  =\int_0^1 \frac{f(t)}{1-tz}\,\mathrm{d}t.
\]
The natural boundedness range is $p>\alpha+2$.  Write
\[
   a=\frac{\alpha+2}{p}\in(0,1).
\]

In 2018, Karapetrovi\'c proposed the norm identity \cite{Karapetrovic2018}
\begin{equation}
  \label{eq:K-formula}
  \norm{\Hcal}_{A^p_\alpha\to A^p_\alpha}
  =\B(a,1-a)
  =\frac{\pi}{\sin(\pi a)}
  =\frac{\pi}{\sin((\alpha+2)\pi/p)}.
\end{equation}
The boundary singularity already determines the usual lower bound.  Indeed, $(1-z)^{-a}$ has the critical homogeneity at $1$, and the integral representation of $\Hcal$ gives
\[
   \int_0^1 t^{a-1}(1-t)^{-a}\,\mathrm{d}t=\B(a,1-a).
\]
The reverse inequality is the difficult part.  It can fail if a profile away from the single boundary point gives a larger quotient.

For $\alpha=0$, formula \eqref{eq:K-formula} was proved by Bo\v{z}in and Karapetrovi\'c \cite{BozinKarapetrovic2018}.  Karapetrovi\'c obtained the lower bound and several weighted estimates in \cite{Karapetrovic2018,Karapetrovic2021}.  The weighted-composition approach of Lindstr\"om, Miihkinen, and Wikman gives further sharp estimates \cite{LindstromMiihkinenWikman2021}.  More recent parameter ranges were found by Dai \cite{Dai2024} and by Bao, Tian, and Wulan \cite{BaoTianWulan2026}.  A large high-weight region is still beyond the pointwise estimates used in those papers.

Our positive results concern the even exponents.  If $p=2m$, the substitution $G=f^m$ changes the upper estimate into a square-function problem on $A^2_\alpha$.  After truncation, Schur's test applies to a finite positive matrix.  The required row inequalities are hypergeometric moment inequalities.  This gives the beta value for $p=2$ and for a range which is uniform in $m$; for $p=4,6,8,10$, the resulting estimates cover every admissible weight.

There is also an obstruction to a global formula.  The test function $(1-z^2)^{-4/5}$, evaluated near the boundary, gives a quotient larger than the beta value along the ray
\[
 a=a_0=\frac{800001}{1000000},
\]
for all real $p\ge1100000$.  The endpoint comes from an interval estimate and is not expected to be optimal.  Notice that this ray lies outside the uniform positive range above.  In particular, the counterexample says nothing about the unsettled part of $p=12$.

\subsection{Main results}
We state the counterexample first.

\begin{theorem}
\label{thm:counterexample}
Put
\[
 a_0=\frac{800001}{1000000},
 \qquad p_0=1100000.
\]
For every real $p\ge p_0$, let
\[
 \alpha_p=a_0p-2.
\]
Then $(p,\alpha_p)$ is admissible and
\[
 \norm{\Hcal}_{A_{\alpha_p}^{p}\to A_{\alpha_p}^{p}}
 >
 \B(a_0,1-a_0).
\]
In particular, \eqref{eq:K-formula} fails for every even exponent
\[
 p=2m,\qquad m\ge550000,
\]
at the corresponding weight $\alpha_{2m}=2a_0m-2$.
\end{theorem}

The number $1100000$ is only a convenient rigorous threshold.  We do not know the least real, or even, exponent for which the formula fails.

We next record the positive results, beginning with the Hilbertian case.

\begin{theorem}
\label{thm:p2}
For every $-1<\alpha<0$,
\[
  \norm{\Hcal}_{A^2_\alpha\to A^2_\alpha}
  =\B\!\left(\frac{\alpha+2}{2},1-\frac{\alpha+2}{2}\right)
  =\frac{\pi}{\sin((\alpha+2)\pi/2)}.
\]
\end{theorem}

The $p=2$ proof uses two Schur inequalities on $\ell^2$.  The first follows from log-convexity of the gamma function.  The second is a hypergeometric moment inequality, proved by monotonicity.

The same method gives a range that is uniform in $m$.

\begin{theorem}
\label{thm:uniform}
Let $m\ge2$, $p=2m$, $\alpha>-1$, and
\[
   a=\frac{\alpha+2}{2m}.
\]
If
\[
   0<a\le \frac{m}{2m-1},
\]
then
\[
  \norm{\Hcal}_{A^{2m}_\alpha\to A^{2m}_\alpha}
  =\B(a,1-a)
  =\frac{\pi}{\sin(\pi a)}.
\]
Equivalently,
\[
  -1<\alpha\le \frac{2(m-1)^2}{2m-1}.
\]
\end{theorem}

For the first four non-Hilbertian even exponents, no restriction beyond admissibility is needed.

\begin{theorem}
\label{thm:small-even}
For $m=2,3,4,5$ (that is, $p=4,6,8,10$) and every
\[
   -1<\alpha<2m-2,
\]
one has
\[
  \norm{\Hcal}_{A^{2m}_\alpha\to A^{2m}_\alpha}
  =\B\!\left(\frac{\alpha+2}{2m},1-\frac{\alpha+2}{2m}\right)
  =\frac{\pi}{\sin((\alpha+2)\pi/(2m))}.
\]
\end{theorem}

The hard intervals in this theorem require one-sign-change and beta/trigonometric positivity estimates.  The longer elementary verifications appear in Appendix \ref{app:hard-small}.

We will also use the following local criterion for a fixed parameter pair.

\begin{theorem}
\label{thm:moving}
Fix $m\ge2$ and $a\in(0,1)$ such that $\alpha=2ma-2>-1$.  Suppose that every finite-section Boyd extremizer $c>0$ for this parameter pair admits a polynomial $P_c$, with $\deg P_c\le m-1$, such that the sequence
\[
   h_{c,j}=[z^j](1-z)^{-m}P_c(1-z)
\]
is positive and satisfies
\[
   \mathcal L_c h_c\le \B(a,1-a)D h_c.
\]
Then
\[
 \norm{\Hcal}_{A^{2m}_{2ma-2}\to A^{2m}_{2ma-2}}
 =\B(a,1-a).
\]
If the hypothesis holds for every $a$ in an interval, the same conclusion holds throughout the corresponding interval of weights.
\end{theorem}

Here $D$ is the Bergman coefficient diagonal and $\mathcal L_c$ is the nonlinear Boyd--Perron matrix defined in Section \ref{sec:boyd-perron}.  The counterexample shows that the hypothesis cannot hold at $a=a_0$ when $2m\ge1100000$.

\subsection*{Organization of the paper}
Section \ref{sec:method} contains the Bergman-space normalizations, the sharp lower bound, and the reduction of the even-exponent problem to finite positive matrices.  The same section recalls the versions of Schur's test and Boyd's variational principle used later, and explains how the resulting Perron equation differs from Boyd's original integral-operator setting.  Section \ref{sec:proofs} proves the positive norm results.  The real-exponent counterexample and its interval verification are given in Section \ref{sec:counterexample}.  Section \ref{sec:scope} records the remaining $p=12$ calculation and a local moment criterion.  Standard hypergeometric formulas and the longer sign computations for $m=2,3,4,5$, together with the partial calculation for $m=6$, are collected in Appendices \ref{app:identities}--\ref{app:m6}.
\section{Preliminaries and the square-function reduction}
\label{sec:method}

\subsection{Basic Bergman normalizations}

We begin with the monomial normalization.  With normalized area measure,
\[
   \int_\D |z|^{2n}\,\mathrm{d}A_\alpha(z)
   =2(\alpha+1)\int_0^1 r^{2n+1}(1-r^2)^\alpha\,\mathrm{d}r.
\]
After the change of variables $u=r^2$, this becomes
\[
   (\alpha+1)\int_0^1 u^n(1-u)^\alpha\,\mathrm{d}u
   =\Gamma(\alpha+2)\frac{\Gamma(n+1)}{\Gamma(n+\alpha+2)}.
\]
Thus, if $\gamma=\alpha+2$, then
\begin{equation}
\label{eq:monomial-norm}
   \norm{z^n}_{A^2_{\gamma-2}}^2=\frac{n!}{(\gamma)_n}.
\end{equation}
These monomial norms form the diagonal matrix $D$ used below.

We shall also use the following boundary singularity criterion.

\begin{lemma}
\label{lem:singularity}
Let $\alpha>-1$, $0<p<\infty$, and $\sigma>0$.  Then
\[
   (1-z)^{-\sigma}\in A^p_\alpha
   \quad\Longleftrightarrow\quad
   p\sigma<\alpha+2.
\]
\end{lemma}

\begin{proof}
Only a small neighborhood of $1\in\partial\D$ matters.  Write $z=re^{i\theta}$ and $x=1-r$.  For $r$ close to $1$ and $|\theta|$ small,
\[
   1-r^2\asymp x,
   \qquad
   |1-re^{i\theta}|^2\asymp x^2+\theta^2.
\]
Thus, the local integral is comparable to
\[
  \int_0^\delta\int_{-\delta}^{\delta}
  x^\alpha(x^2+\theta^2)^{-p\sigma/2}\,\mathrm{d}\theta\,\mathrm{d}x.
\]
Passing to polar coordinates in the half-plane $x>0$ gives a radial power
\[
   \rho^{\alpha-p\sigma+1}.
\]
The integral converges at $0$ exactly when
\[
   \alpha-p\sigma+1>-1,
\]
which is equivalent to $p\sigma<\alpha+2$.
\end{proof}

\begin{remark}
The same local computation also shows the mass concentration used in the lower-bound argument: if $\sigma=a-\varepsilon$ with $a=(\alpha+2)/p$, then the norm of $(1-z)^{-\sigma}$ diverges like a positive multiple of $\varepsilon^{-1/p}$ as $\varepsilon\downarrow0$, and the mass concentrates in every fixed neighborhood of the point $1$.
\end{remark}

\subsection{The sharp lower bound}

\begin{proposition}
\label{prop:lower}
Let $\alpha>-1$, $p>\alpha+2$, and $a=(\alpha+2)/p$.  If $\Hcal$ is bounded on $A^p_\alpha$, then
\[
   \norm{\Hcal}_{A^p_\alpha\to A^p_\alpha}\ge \B(a,1-a).
\]
\end{proposition}

\begin{proof}
For $0<\varepsilon<a$, put
\[
  f_\varepsilon(z)=(1-z)^{-a+\varepsilon}.
\]
By Lemma \ref{lem:singularity}, $f_\varepsilon\in A^p_\alpha$, since
\[
  p(a-\varepsilon)=\alpha+2-p\varepsilon<\alpha+2.
\]
Using the weighted-composition identity
\[
  \Hcal f(z)=\int_0^1\frac{1}{1-(1-t)z}
  f\!\left(\frac{t}{1-(1-t)z}\right)\,\mathrm{d}t,
\]
one obtains
\[
  \Hcal f_\varepsilon(z)=f_\varepsilon(z)G_\varepsilon(z),
\]
where
\[
  G_\varepsilon(z)
  =\int_0^1(1-t)^{-a+\varepsilon}
    (1-(1-t)z)^{a-\varepsilon-1}\,\mathrm{d}t.
\]
As $z\to1$ inside $\D$,
\[
   G_\varepsilon(z)\to \B(a-\varepsilon,1-a+\varepsilon).
\]
Moreover, as $\varepsilon\downarrow0$, the measures
\[
   |f_\varepsilon(z)|^p\,\mathrm{d}A_\alpha(z)
\]
concentrate at the boundary point $1$.  Hence, \[
  \liminf_{\varepsilon\downarrow0}\frac{\norm{\Hcal f_\varepsilon}_{A^p_\alpha}}{\norm{f_\varepsilon}_{A^p_\alpha}}
  \ge \lim_{\varepsilon\downarrow0}\B(a-\varepsilon,1-a+\varepsilon)
  =\B(a,1-a).
\]
This proves the lower bound.
\end{proof}

\subsection{The even-exponent square-function reduction}

Let $p=2m$ and set
\[
   \gamma=\alpha+2=2ma,
   \qquad a\in(0,1).
\]
For $G\in A^2_{\gamma-2}$, define
\[
  S_y^{(m)}G(z)
  =\frac{1}{(1-(1-y)z)^m}
  G\!\left(\frac{y}{1-(1-y)z}\right),
  \qquad 0<y<1.
\]

\begin{lemma}
\label{lem:square-reduction}
Suppose that for every analytic polynomial $G$,
\begin{equation}
\label{eq:SF}
  \int_0^1\norm{S_y^{(m)}G}_{A^2_{\gamma-2}}^{1/m}\,\mathrm{d}y
  \le \B(a,1-a)\norm{G}_{A^2_{\gamma-2}}^{1/m}.
\end{equation}
Then
\[
  \norm{\Hcal}_{A^{2m}_{\gamma-2}\to A^{2m}_{\gamma-2}}
  \le \B(a,1-a).
\]
Together with Proposition \ref{prop:lower}, this yields the sharp norm formula.
\end{lemma}

\begin{proof}
Write
\[
  \Hcal f(z)=\int_0^1T_yf(z)\,\mathrm{d}y,
\]
where
\[
  T_yf(z)= \frac{1}{1-(1-y)z}
  f\!\left(\frac{y}{1-(1-y)z}\right).
\]
If $G=f^m$, then
\[
  S_y^{(m)}G=(T_yf)^m.
\]
Therefore, \[
  \norm{T_yf}_{A^{2m}_{\gamma-2}}
  =\norm{S_y^{(m)}G}_{A^2_{\gamma-2}}^{1/m},
  \qquad
  \norm{G}_{A^2_{\gamma-2}}^{1/m}
  =\norm{f}_{A^{2m}_{\gamma-2}}.
\]
Minkowski's inequality gives
\[
  \norm{\Hcal f}_{A^{2m}_{\gamma-2}}
  \le \int_0^1\norm{T_yf}_{A^{2m}_{\gamma-2}}\,\mathrm{d}y
  =\int_0^1\norm{S_y^{(m)}(f^m)}_{A^2_{\gamma-2}}^{1/m}\,\mathrm{d}y.
\]
Applying \eqref{eq:SF} to $G=f^m$ proves the upper bound for polynomials, and density completes the proof.
\end{proof}

This is the only point where the assumption $p=2m$ is used: $G=f^m$ turns a $2m$-norm into a square norm.  The polynomial reduction causes no loss.  Indeed, $f_r(z)=f(rz)$ tends to $f$ in $A^p_\alpha$ as $r\uparrow1$, and $f_r$ may be approximated uniformly on $\overline\D$ by its Taylor polynomials.

\subsection{Finite matrices}

Let
\[
  G(z)=\sum_{j=0}^N c_jz^j.
\]
By \eqref{eq:monomial-norm},
\[
  \norm{z^j}_{A^2_{\gamma-2}}^2=\frac{j!}{(\gamma)_j},
\]
and we set
\[
  D_j=\frac{j!}{(\gamma)_j}.
\]
The matrix $D=\operatorname{diag}(D_0,D_1,\dots)$ is the coefficient form of the $A^2_{\gamma-2}$ norm.  In particular,
\[
   \norm{G}_{A^2_{\gamma-2}}^2=c^*Dc.
\]
The coefficient expansion
\[
S_y^{(m)}G(z)=\sum_{n=0}^{\infty}(1-y)^n
\left[\sum_{j=0}^Nc_jy^j\binom{m+j+n-1}{n}\right]z^n
\]
gives
\[
  \norm{S_y^{(m)}G}_{A^2_{\gamma-2}}^2=c^*B_yc,
\]
where
\begin{equation}
\label{eq:By}
  (B_y)_{jk}=y^{j+k}\sum_{n=0}^{\infty}
  \frac{(m+j)_n(m+k)_n}{(\gamma)_n n!}(1-y)^{2n}.
\end{equation}

\begin{proof}[Derivation of \eqref{eq:By}]
The coefficient calculation is short.  Since
\[
  (1-(1-y)z)^{-m-j}
  =\sum_{n=0}^{\infty}\frac{(m+j)_n}{n!}(1-y)^nz^n,
\]
we have
\[
 S_y^{(m)}(z^j)
 =y^j(1-(1-y)z)^{-m-j}
 =\sum_{n=0}^{\infty}y^j\frac{(m+j)_n}{n!}(1-y)^nz^n.
\]
Thus, for $G(z)=\sum_jc_jz^j$ the coefficient of $z^n$ in $S_y^{(m)}G$ is
\[
   (1-y)^n\sum_j c_jy^j\frac{(m+j)_n}{n!}.
\]
Multiplying the coefficient by the monomial norm $n!/(\gamma)_n$ and summing in $n$ gives
\[
 \sum_{n=0}^{\infty}\frac{n!}{(\gamma)_n}(1-y)^{2n}
 \left|\sum_jc_jy^j\frac{(m+j)_n}{n!}\right|^2.
\]
Expanding the square yields exactly \eqref{eq:By}.
\end{proof}

Thus, \eqref{eq:SF} is equivalent to
\begin{equation}
\label{eq:finite-target}
  \int_0^1(c^*B_yc)^{1/(2m)}\,\mathrm{d}y
  \le \B(a,1-a)(c^*Dc)^{1/(2m)}.
\end{equation}

For a polynomial, the $j$-sum is finite and the $n$-sum is absolutely convergent.  Uniform estimates for these finite sections pass to every $G\in A^2_{\gamma-2}$ by Taylor truncation and Fatou's lemma.

\subsection{Tangent reduction and weighted Schur tests}

Let
\[
  \psi(y)=y^{a-1}(1-y)^{-a}.
\]
Then
\[
  \int_0^1\psi(y)\,\mathrm{d}y=\B(a,1-a).
\]
For $q=1/(2m)$ and $X,\Lambda>0$,
\[
  X^q\le q\Lambda^{q-1}X+(1-q)\Lambda^q.
\]
Choosing
\[
  X=c^*B_yc,
  \qquad
  \Lambda=\psi(y)^{2m}c^*Dc,
\]
shows that \eqref{eq:finite-target} follows from
\begin{equation}
\label{eq:linear-target}
  \int_0^1 \psi(y)^{1-2m} c^*B_yc\,\mathrm{d}y
  \le \B(a,1-a)c^*Dc.
\end{equation}
Define
\begin{equation}
\label{eq:Lm-def}
  L_{jk}=\int_0^1 y^{(2m-1)(1-a)}(1-y)^{(2m-1)a}(B_y)_{jk}\,\mathrm{d}y.
\end{equation}
Then \eqref{eq:linear-target} is $L\le \B(a,1-a)D$ as quadratic forms.

To see that the constant has not changed, use \eqref{eq:linear-target} in the tangent inequality:
\[
\begin{aligned}
\int_0^1(c^*B_yc)^q\,\mathrm{d}y
&\le q(c^*Dc)^{q-1}\int_0^1\psi(y)^{1-2m}c^*B_yc\,\mathrm{d}y\\
&\quad +(1-q)(c^*Dc)^q\int_0^1\psi(y)\,\mathrm{d}y \\
&\le \B(a,1-a)(c^*Dc)^q.
\end{aligned}
\]

\subsection{Schur's test and its use in this paper}

Schur introduced positive auxiliary sequences in his study of bounded infinite bilinear forms \cite{Schur1911}; the classical Hilbert inequality, with its sharp constant $\pi$, is one of the main examples behind the method.  A modern account is given by Garcia, Mashreghi, and Ross.  They first treat the one-weight case in \cite[Section 3.2, pp.~74--76]{GarciaMashreghiRoss2023}, and then prove the two-weight form in \cite[Theorem 3.3.1, p.~77]{GarciaMashreghiRoss2023}.  We state the latter because the $p=2$ argument below is not symmetric in the standard coefficient basis.

\begin{schurtest}
Let $A=(a_{ij})_{i,j\ge0}$ be an infinite matrix.  Suppose that there are positive constants $\alpha,\beta$ and positive sequences $(p_i)_{i\ge0}$ and $(q_i)_{i\ge0}$ such that
\begin{equation}
\label{eq:classical-schur}
   \sum_{i=0}^{\infty}|a_{ij}|p_i\le \alpha q_j,\qquad j\ge0,
   \qquad
   \sum_{j=0}^{\infty}|a_{ij}|q_j\le \beta p_i,\qquad i\ge0.
\end{equation}
Then $A$ defines a bounded operator on $\ell^2$ and
\[
    \norm{A}_{\ell^2\to\ell^2}\le \sqrt{\alpha\beta}.
\]
\end{schurtest}

There are three slightly different uses of this test in the paper.
\begin{enumerate}[label=(\roman*)]
\item For $p=2$, conjugating by the Bergman monomial norms produces the matrix $K$ in \eqref{eq:p2-K}.  It is not symmetric in the standard $\ell^2$ basis.  Both inequalities in \eqref{eq:classical-schur} are therefore needed, with two different gamma-function weights; see \eqref{eq:p2-schur1} and \eqref{eq:p2-schur2}.
\item For $p=2m$, the tangent estimate produces a symmetric matrix $L$ together with the diagonal Bergman metric $D$.  Conjugation by $D^{1/2}$ turns the problem into an ordinary symmetric Schur estimate.  Only one positive sequence is then required.
\item Schur's test supplies the upper bound only.  Sharpness comes from the boundary family in Proposition \ref{prop:lower}.  Thus no equality case of the Schur test is needed.  What matters is to find a positive supersolution whose row sums have the beta constant.
\end{enumerate}

The form needed for the even-exponent argument is recorded next.

\begin{lemma}
\label{lem:weighted-schur}
Let $L=(L_{jk})_{j,k\ge0}$ be a symmetric matrix with nonnegative entries, and let $D=\operatorname{diag}(D_j)$ with $D_j>0$.  If there exists $h_j>0$ and $C>0$ such that
\[
  \sum_{k\ge0}L_{jk}h_k\le C D_jh_j,
  \qquad j=0,1,2,\dots,
\]
then
\[
  c^*Lc\le C c^*Dc
\]
for all finitely supported vectors $c$.
\end{lemma}

This is the one-weight Schur test of \cite[Section 3.2, pp.~74--76]{GarciaMashreghiRoss2023}, applied to $D^{-1/2}LD^{-1/2}$ with the weight $D^{1/2}h$.

The lemma is used uniformly on finite sections.  This point is important: the Schur vector may depend on $m$ and $a$, but the constant and the row inequalities must not depend on the truncation index.  Taylor truncation and Fatou's lemma then give the infinite-dimensional estimate.

\subsection{The Boyd--Perron interpretation}
\label{sec:boyd-perron}

Boyd's paper \cite{Boyd1969} studies best constants for functionals
\[
  J(f)=\int |Tf|^\rho |f|^q\,\mathrm d\mu
\]
under an $L^r$ normalization, where $T$ is a positive compact integral operator and $s=\rho r/(r-q)$.  The argument has three distinct parts.  Lemma~1 of \cite[pp.~369--371]{Boyd1969} proves that the supremum is attained.  Lemma~2 of \cite[pp.~371--375]{Boyd1969} proves positivity of an extremizer under a strict positivity assumption on the kernel, derives the Euler equation, and identifies its largest normalized eigenvalue.  Finally, Theorem~1 of \cite[pp.~368--376]{Boyd1969} specializes the abstract result to a Volterra operator and converts the Euler equation into a nonlinear boundary-value problem.  Only the first two ideas are relevant here.

For comparison with the finite calculation below, we record the abstract variational statement.  This is the part of Boyd's Lemmas~1 and~2 that we use as a model.

\begin{theoremA}
Let $\rho>0$, $r>1$, $0\le q<r$, and put $s=\rho r/(r-q)$.  Let $T$ be a positive compact integral operator from $L^r_m$ to $L^s_m$, and define
\[
   J(f)=\int |Tf|^\rho |f|^q\,\mathrm{d}\mu,
   \qquad
   K^*=\sup\{J(f):\norm{f}_{L^r_m}\le1\}.
\]
Then there exists $f_0\ge0$ with $\norm{f_0}_{L^r_m}=1$ and $J(f_0)=K^*$.  Under Boyd's strict positivity hypothesis, an extremizer is positive almost everywhere and satisfies
\begin{equation}
\label{eq:boyd-eigen}
 r\lambda f^{r-1}
 -q(Tf)^\rho f^{q-1}
 -\rho T^*\big((Tf)^{\rho-1}f^q\big)=0,
 \qquad
 \lambda=\frac{K^*(\rho+q)}{r}.
\end{equation}
This $\lambda$ is the largest value for which a normalized positive solution of \eqref{eq:boyd-eigen} exists.
\end{theoremA}

Our use of this principle is finite-dimensional and does not involve Boyd's Volterra reduction or its differential equation.  Fix a section $0\le j,k\le N$ and write
\[
 \Phi_N(c)=\int_0^1(c^*B_yc)^{1/(2m)}\,\mathrm{d}y,
 \qquad
 \Lambda_N=\max_{c^*Dc=1}\Phi_N(c).
\]
The ellipsoid $c^*Dc=1$ is compact, so an extremizer exists without an operator-compactness argument.  Since every entry of $B_y$ is nonnegative,
\[
  \Phi_N(|c|)\ge \Phi_N(c),
  \qquad |c|^*D|c|=c^*Dc;
\]
hence an extremizer may be chosen nonnegative.  Put
\[
  R_c(y)=c^*B_yc.
\]
For a nonzero nonnegative $c$, one has $R_c(y)>0$ on $(0,1)$.  The functional is therefore differentiable at $c$.  The Lagrange multiplier equation, followed by multiplication by $c^*$, gives
\begin{equation}
\label{eq:EL-general}
  \int_0^1 R_c(y)^{1/(2m)-1}B_yc\,\mathrm{d}y
  =\Lambda_NDc.
\end{equation}
Thus, with
\[
  \mathcal L_c
  =\int_0^1R_c(y)^{1/(2m)-1}B_y\,\mathrm{d}y,
\]
we have
\begin{equation}
\label{eq:perron-general}
  \mathcal L_cc=\Lambda_NDc.
\end{equation}
All entries of $\mathcal L_c$ are positive.  Equation \eqref{eq:perron-general} consequently shows that $c$ is strictly positive.  Moreover,
\[
  A_c=D^{-1/2}\mathcal L_cD^{-1/2}
\]
is a positive symmetric matrix, and $D^{1/2}c$ is its positive eigenvector.  The Perron--Frobenius theorem \cite[Section 8.2]{HornJohnson2013} gives
\[
  \Lambda_N=\rho(A_c)=\rho(D^{-1}\mathcal L_c).
\]
For every $h>0$, the Collatz--Wielandt upper estimate \cite[Section 8.2]{HornJohnson2013} now reads
\begin{equation}
\label{eq:collatz-bound}
 \Lambda_N
 \le
 \max_{0\le j\le N}
 \frac{(\mathcal L_ch)_j}{D_jh_j}.
\end{equation}
The same bound also follows at once from Lemma \ref{lem:weighted-schur}.

This is where the present argument departs from Boyd's original setting.
\begin{enumerate}[label=(\roman*)]
\item Boyd starts from a fixed integral operator $T$.  Here the positive matrix $\mathcal L_c$ depends on the unknown extremizer through the factor $R_c(y)^{1/(2m)-1}$.
\item Boyd uses compactness and kernel positivity to obtain an extremizer and then studies a nonlinear differential equation.  On each finite section, compactness and strict positivity are elementary, and the Euler equation is the matrix equation \eqref{eq:perron-general}.
\item We do not solve the nonlinear eigenvalue equation.  Instead, we bound its Perron root by constructing a positive supersolution in \eqref{eq:collatz-bound}.  For Theorems \ref{thm:uniform} and \ref{thm:small-even}, the same vector $h_j=(m)_j/j!$ works for every extremizer.  In Theorem \ref{thm:moving}, the supersolution is allowed to depend on $c$.
\item The resulting estimate must be uniform in $N$.  This uniformity, rather than the existence of an extremizer on one finite section, is what permits passage back to the Bergman space.
\end{enumerate}

In particular, if $h_c>0$ satisfies
\[
  \mathcal L_ch_c\le \B(a,1-a)Dh_c,
\]
then \eqref{eq:collatz-bound} yields $\Lambda_N\le\B(a,1-a)$.  A useful class of moving supersolutions is
\[
   h_c(z)=(1-z)^{-m}P_c(1-z),
   \qquad \deg P_c\le m-1.
\]
This is the condition used in Theorem \ref{thm:moving}.

\subsection{The fixed Schur vector}

For the estimates below we take
\[
   h_j=\frac{(m)_j}{j!},
   \qquad
   h(z)=\sum_{j=0}^{\infty}h_jz^j=(1-z)^{-m}.
\]
Then the row action of the matrices $B_y$ simplifies because
\[
  \sum_{k=0}^{\infty}(m+k)_n\frac{(m)_k}{k!}y^k=(m)_n(1-y)^{-m-n}.
\]
Consequently,
\[
   (B_yh)_j
   =y^j(1-y)^{-m}{}_2F_1(j+m,m;\gamma;1-y).
\]
Insert the tangent weight and apply Euler's formula.  The Schur row inequality then reduces to a one-variable moment estimate, and the vector $(m)_j/j!$ appears naturally.  For larger $m$, this fixed vector may be too crude even when the beta formula is true.  At a single parameter pair one can instead try
\[
   h_P(z)=(1-z)^{-m}P(1-z),
   \qquad \deg P\le m-1.
\]

\section{Proofs of the main results}
\label{sec:proofs}

\subsection{Proof of Theorem \ref{thm:p2}}

Let $-1<\alpha<0$, $p=2$, and
\[
   a=\frac{\alpha+2}{2}\in\left(\frac{1}{2},1\right).
\]
Then $\alpha=2a-2$.  The monomials are orthogonal in $A^2_\alpha$, and
\[
  \omega_n:=\norm{z^n}_{A^2_\alpha}^2
  =\frac{\Gamma(2a)\Gamma(n+1)}{\Gamma(n+2a)}.
\]
The unitary map
\[
  f(z)=\sum_{n\ge0}c_nz^n
  \longmapsto
  x_n=\sqrt{\omega_n}\,c_n
\]
sends $A^2_\alpha$ onto $\ell^2$.  Under this map, $\Hcal$ is represented by the positive matrix
\begin{equation}
\label{eq:p2-K}
  K_{nk}=\frac{1}{n+k+1}\left(\frac{\omega_n}{\omega_k}\right)^{1/2}.
\end{equation}
It is enough to prove
\[
  \norm{K}_{\ell^2\to\ell^2}\le \B(a,1-a).
\]

Take the Schur weight
\[
  h_n=\sqrt{\omega_n}\frac{(a)_n}{n!}.
\]

\begin{remark}
The coefficients of the boundary profile $(1-z)^{-a}$ are $(a)_n/n!$.  Passing to the standard $\ell^2$ basis contributes the factor $\sqrt{\omega_n}$.  Hence $h_n$ is simply the boundary coefficient vector in orthonormal coordinates, which accounts for the beta constant in the Schur sums.
\end{remark}

We prove
\begin{equation}
\label{eq:p2-schur1}
  \sum_{k=0}^{\infty}K_{nk}h_k\le \B(a,1-a)h_n,
\end{equation}
and
\begin{equation}
\label{eq:p2-schur2}
  \sum_{n=0}^{\infty}K_{nk}h_n\le \B(a,1-a)h_k.
\end{equation}

For \eqref{eq:p2-schur1},
\[
\sum_{k=0}^{\infty}K_{nk}h_k
=\sqrt{\omega_n}\sum_{k=0}^{\infty}\frac{(a)_k}{k!(n+k+1)}.
\]
Since
\[
  \sum_{k=0}^{\infty}\frac{(a)_k}{k!}x^k=(1-x)^{-a},
\]
we get
\[
\sum_{k=0}^{\infty}\frac{(a)_k}{k!(n+k+1)}
=\int_0^1x^n(1-x)^{-a}\,\mathrm{d}x
=\B(n+1,1-a).
\]
Thus, \eqref{eq:p2-schur1} is equivalent to
\[
  \frac{\B(n+1,1-a)}{(a)_n/n!}\le \B(a,1-a).
\]
But
\[
 \frac{\B(n+1,1-a)}{(a)_n/n!}
 =\B(a,1-a)\frac{\Gamma(n+1)^2}{\Gamma(n+a)\Gamma(n+2-a)}.
\]
Since
\[
  n+1=\frac{n+a+n+2-a}{2},
\]
The log-convexity of the gamma function gives
\[
  \Gamma(n+1)^2\le\Gamma(n+a)\Gamma(n+2-a).
\]
This proves \eqref{eq:p2-schur1}.

For \eqref{eq:p2-schur2}, we need the following lemma.

\begin{lemma}
\label{lem:p2-moment}
Let $1/2<a<1$.  Then for every $k\ge0$,
\[
\sum_{n=0}^{\infty}\frac{\Gamma(n+a)}{\Gamma(n+2a)}\frac{1}{n+k+1}
\le
\B(a,1-a)\frac{\Gamma(k+a)}{\Gamma(k+2a)}.
\]
\end{lemma}

\begin{proof}
Set
\[
 A_n=\frac{\Gamma(n+a)}{\Gamma(n+2a)},
 \qquad
 D_k=\frac{\Gamma(k+a)}{\Gamma(k+2a)},
\]
\[
 S_k=\sum_{n=0}^{\infty}\frac{A_n}{n+k+1},
 \qquad
 R_k=\frac{S_k}{D_k}.
\]
We prove that $R_k$ is increasing and that $R_k\to \B(a,1-a)$.  It then follows that $R_k\le\B(a,1-a)$.

Let
\[
  F(x)=\sum_{n=0}^{\infty}A_nx^n.
\]
Then
\[
  S_k=\int_0^1x^kF(x)\,\mathrm{d}x.
\]
Moreover, \[
  F(x)=\frac{\Gamma(a)}{\Gamma(2a)}{}_2F_1(a,1;2a;x).
\]
Euler's transformation gives
\[
  F(x)=\frac{\Gamma(a)}{\Gamma(2a)}(1-x)^{a-1}\Theta(x),
\]
where
\[
  \Theta(x)={}_2F_1(a,2a-1;2a;x).
\]
Since $a>1/2$, the Taylor coefficients of $\Theta$ are nonnegative, so $\Theta$ is increasing on $(0,1)$.

Because
\[
  \frac{D_{k+1}}{D_k}=\frac{k+a}{k+2a},
\]
the inequality $R_{k+1}\ge R_k$ is equivalent to
\[
  (k+2a)S_{k+1}-(k+a)S_k\ge0.
\]
The left-hand side is a positive constant times
\[
  \int_0^1 \Theta(x)g_k(x)\,\mathrm{d}x,
\]
where
\[
  g_k(x)=x^k((k+2a)x-(k+a))(1-x)^{a-1}.
\]
The function $g_k$ changes sign once, at
\[
  x_k=\frac{k+a}{k+2a},
\]
from negative to positive.  Also, \[
\int_0^1g_k(x)\,\mathrm{d}x
=(k+2a)\B(k+2,a)-(k+a)\B(k+1,a)
=\frac{a(1-a)}{k+a+1}\B(k+1,a)>0.
\]
Since $\Theta$ is increasing,
\[
\int_0^1\Theta(x)g_k(x)\,\mathrm{d}x
=\int_0^1(\Theta(x)-\Theta(x_k))g_k(x)\,\mathrm{d}x
 +\Theta(x_k)\int_0^1g_k(x)\,\mathrm{d}x\ge0.
\]
Thus, $R_k$ is increasing.

Finally, $A_n\sim n^{-a}$ and $D_k\sim k^{-a}$.  A standard Riemann-sum argument gives
\[
  S_k\sim k^{-a}\int_0^\infty \frac{t^{-a}}{1+t}\,\mathrm{d}t
  =k^{-a}\B(1-a,a).
\]
Hence, \[
  \lim_{k\to\infty}R_k=
  \B(1-a,a)=\B(a,1-a).
\]
The lemma follows.
\end{proof}

Now
\[
\sum_{n=0}^{\infty}K_{nk}h_n
=\frac{1}{\sqrt{\omega_k}}\sum_{n=0}^{\infty}\frac{\omega_n(a)_n/n!}{n+k+1}.
\]
Since
\[
  \omega_n\frac{(a)_n}{n!}
  =\frac{\Gamma(2a)}{\Gamma(a)}\frac{\Gamma(n+a)}{\Gamma(n+2a)},
\]
Lemma \ref{lem:p2-moment} gives
\[
\sum_{n=0}^{\infty}K_{nk}h_n
\le
\frac{\Gamma(2a)}{\Gamma(a)\sqrt{\omega_k}}
\B(a,1-a)\frac{\Gamma(k+a)}{\Gamma(k+2a)}.
\]
Using
\[
  \omega_k=\frac{\Gamma(2a)\Gamma(k+1)}{\Gamma(k+2a)},
\]
the last expression equals
\[
  \B(a,1-a)\sqrt{\omega_k}\frac{\Gamma(k+a)}{\Gamma(a)\Gamma(k+1)}
  =\B(a,1-a)h_k.
\]
Thus, \eqref{eq:p2-schur2} holds.  Schur's test gives
\[
  \norm{\Hcal}_{A^2_\alpha\to A^2_\alpha}\le\B(a,1-a).
\]
The reverse inequality is Proposition \ref{prop:lower}.  This proves Theorem \ref{thm:p2}.

\subsection{Proof of Theorem \ref{thm:uniform}}

Let $m\ge2$, $p=2m$, $\gamma=2ma$, and $\alpha=\gamma-2$.  By Lemma \ref{lem:square-reduction}, it suffices to prove \eqref{eq:SF}.  For $a\le1/2$, the usual affine-composition estimate gives the result.  We recall the point briefly.  In the horodisk form one writes the relevant operator as
\[
   G\longmapsto \Phi_r^{\,m(1-2a)}G\circ\Phi_r,
   \qquad \Phi_r(z)=1-r+rz,
\]
with $0<r<1/2$.  If $a\le1/2$, then the exponent $m(1-2a)$ is nonnegative and $|\Phi_r(z)|\le1$ on $\D$.  Hence, the multiplier has modulus at most one.  The standard affine composition estimate on $A^2_{\gamma-2}$ then gives the factor $r^{-a}$, and the remaining integral is exactly
\[
   \int_0^{1/2}\frac{(1-2r)^{a-1}}{1-r}r^{-a}\,\mathrm{d}r=
   \int_0^1 y^{a-1}(1-y)^{-a}\,\mathrm{d}y=\B(a,1-a).
\]
Thus, only the hard half $a>1/2$ requires the Schur certificate below.  We therefore assume
\[
   \frac{1}{2}<a\le \frac{m}{2m-1}.
\]

We use the linear tangent reduction and the weighted Schur test.  Take
\[
  h_j=\frac{(m)_j}{j!}.
\]
We compute
\[
\sum_{k=0}^{\infty}(B_y)_{jk}h_k
=y^j(1-y)^{-m}{}_2F_1(j+m,m;\gamma;1-y).
\]
Thus, \begin{equation}
\label{eq:row-general}
(Lh)_j=
\int_0^1 y^{j+(2m-1)(1-a)}(1-y)^{(2m-1)a-m}
{}_2F_1(j+m,m;\gamma;1-y)\,\mathrm{d}y.
\end{equation}
Since $a>1/2$, $\gamma=2ma>m$, and Euler's integral representation applies:
\[
{}_2F_1(j+m,m;\gamma;1-y)
=\frac{\Gamma(\gamma)}{\Gamma(m)\Gamma(\gamma-m)}
\int_0^1\frac{t^{m-1}(1-t)^{\gamma-m-1}}{(1-(1-y)t)^{j+m}}\,\mathrm{d}t.
\]
With the Boyd change of variables
\[
  s=\frac{y}{1-(1-y)t},
\]
\eqref{eq:row-general} is equivalent to the moment inequality
\begin{equation}
\label{eq:M-general}
\int_0^1s^{j+(2m-1)(1-a)}(1-s)^{(2m-1)a-m}J_{m,a}(s)\,\mathrm{d}s
\le
\B(a,1-a)\B(j+m,2ma-m),
\end{equation}
where
\[
  J_{m,a}(s)=\int_0^1\frac{t^{m-1}(1-t)^{a-1}}{1-st}\,\mathrm{d}t.
\]
We claim
\begin{equation}
\label{eq:J-point}
  J_{m,a}(s)\le \B(a,1-a)(1-s)^{a-1}.
\end{equation}
Let $\delta=1-s$ and set $u=(1-t)/\delta$.  Then
\[
J_{m,a}(s)=\delta^{a-1}
\int_0^{1/\delta}\frac{(1-\delta u)^{m-1}u^{a-1}}{1+su}\,\mathrm{d}u.
\]
Since $0\le1-\delta u\le1$ and $m\ge2$,
\[
 (1-\delta u)^{m-1}\le 1-\delta u,
\]
and
\[
 \frac{1-\delta u}{1+su}\le \frac{1}{1+u}.
\]
Therefore, \[
J_{m,a}(s)\le \delta^{a-1}\int_0^\infty\frac{u^{a-1}}{1+u}\,\mathrm{d}u
=\B(a,1-a)(1-s)^{a-1}.
\]
Substituting \eqref{eq:J-point} into \eqref{eq:M-general}, it remains to compare
\[
 \int_0^1s^{j+(2m-1)(1-a)}(1-s)^{2ma-m-1}\,\mathrm{d}s
\]
with
\[
 \int_0^1s^{j+m-1}(1-s)^{2ma-m-1}\,\mathrm{d}s.
\]
This follows from
\[
  (2m-1)(1-a)\ge m-1,
\]
which is exactly $a\le m/(2m-1)$.  Hence, \[
(Lh)_j\le\B(a,1-a)D_jh_j.
\]
By Lemma \ref{lem:weighted-schur}, $L\le \B(a,1-a)D$.  The tangent reduction gives \eqref{eq:SF}.  Proposition \ref{prop:lower} completes the proof of the sharp formula.

\subsection{Proof of Theorem \ref{thm:small-even}}

Let $m=2,3,4,5$.  Theorem \ref{thm:uniform} proves the result for
\[
  0<a\le \frac{m}{2m-1}.
\]
It remains to treat
\[
  \frac{m}{2m-1}<a<1.
\]
We continue with the same Schur weight $h_j=(m)_j/j!$.  The row inequality is still equivalent to \eqref{eq:M-general}.  Define
\[
\mathcal A_{m,a}(s)
  =s^{m-(2m-1)a}(1-s)^{1-a}J_{m,a}(s).
\]
Then \eqref{eq:M-general} is equivalent to
\begin{equation}
\label{eq:Rmj}
\int_0^1s^{j+m-1}(1-s)^{2ma-m-1}
\left[\B(a,1-a)-\mathcal A_{m,a}(s)\right]\,\mathrm{d}s\ge0.
\end{equation}

We first record the sign pattern.  Since
\[
  J_{m,a}(s)=\B(m,a){}_2F_1(1,m;m+a;s),
\]
Euler's transformation gives
\[
  (1-s)^{1-a}J_{m,a}(s)
  =\B(m,a){}_2F_1(a,m+a-1;m+a;s).
\]
Equivalently,
\[
  \mathcal A_{m,a}(s)
  =\frac{\Gamma(m)}{(a)_{m-1}}s^{1-2ma}H(s),
\]
where
\[
  H(s)=\int_0^s u^{a+m-2}(1-u)^{-a}\,\mathrm{d}u.
\]
Let $C=H(1)=\B(a+m-1,1-a)$ and $F(s)=s^{1-2ma}H(s)$.  Then the sign of $\mathcal A_{m,a}(s)-\B(a,1-a)$ is the sign of $F(s)-C$.  A direct computation gives
\[
F'(s)=s^{-2ma}N(s),
\]
where
\[
N(s)=s^{a+m-1}(1-s)^{-a}-(2ma-1)H(s),
\]
and
\[
N'(s)=s^{a+m-2}(1-s)^{-a-1}
\left[m-(2m-1)a+m(2a-1)s\right].
\]
In the range $a>m/(2m-1)$, the bracket changes sign once, from negative to positive.  Moreover, $N(s)<0$ near $0$ and $N(s)\to+\infty$ as $s\to1^-$.  Hence, $F$ first decreases and then increases.  Since $F(s)\to+\infty$ as $s\downarrow0$, $F(1)=C$, and $F(s)<C$ near $1$, we conclude that
\begin{equation}
\label{eq:osc-general}
  \mathcal A_{m,a}(s)-\B(a,1-a)
  \quad\text{has sign pattern}\quad +\to-.
\end{equation}
Thus, \[
  \B(a,1-a)-\mathcal A_{m,a}(s)
\]
has sign pattern $-\to+$.  Consequently, all inequalities \eqref{eq:Rmj} follow from the base moment $j=0$:
\begin{equation}
\label{eq:base-moment}
\int_0^1s^{m-1}(1-s)^{2ma-m-1}
\left[\B(a,1-a)-\mathcal A_{m,a}(s)\right]\,\mathrm{d}s\ge0.
\end{equation}
Indeed, if $s_0$ is the sign-change point, multiplication by $s^j$ weakens the negative part on $(0,s_0)$ and strengthens the positive part on $(s_0,1)$.

Appendix \ref{app:hard-small} verifies \eqref{eq:base-moment} for $m=2,3,4,5$.  It follows that \eqref{eq:Rmj} holds for every $j\ge0$.  Lemma \ref{lem:weighted-schur} now gives $L\le \B(a,1-a)D$, and the tangent reduction yields \eqref{eq:SF}.  The lower bound from Proposition \ref{prop:lower} completes the proof.

\begin{remark}
The sign-change argument \eqref{eq:osc-general} is uniform in $m$.  What depends on $m$ is the base moment \eqref{eq:base-moment}.  Elementary estimates suffice up to $m=5$, but the expressions become much less manageable afterward.  More elaborate certificates may treat additional parameters, although Theorem \ref{thm:counterexample} shows that they cannot give the beta bound everywhere.
\end{remark}

\subsection{Proof of Theorem \ref{thm:moving}}

The finite-section extremal problem and the nonlinear Euler equation are given in Section \ref{sec:boyd-perron}.  If $h_c>0$ satisfies
\[
  \mathcal L_ch_c\le\B(a,1-a)Dh_c,
\]
then, by the generalized Schur/Collatz--Wielandt principle, the Perron eigenvalue in
\[
  \mathcal L_cc=\Lambda_NDc
\]
satisfies
\[
  \Lambda_N\le \B(a,1-a).
\]
Hence, \eqref{eq:finite-target} holds on every finite section with the sharp constant.  Monotone truncation and Fatou's lemma give \eqref{eq:SF} as $N\to\infty$.  Lemma \ref{lem:square-reduction} supplies the upper bound for $\Hcal$, and Proposition \ref{prop:lower} supplies the reverse inequality.

\section{An explicit real-exponent failure half-line}
\label{sec:counterexample}

The same test function will be used for every $p$ in the half-line.  Its form is suggested by the high-power growth-space limit, but the argument itself stays at finite $p$.

\subsection{Point evaluation and a subcritical input}

\begin{lemma}
\label{lem:explicit-evaluation}
If $\alpha>-1$, $p>0$, $h\in A_\alpha^p$, and $z\in\D$, then
\[
 \norm{h}_{A_\alpha^p}
 \ge(1-|z|^2)^{(\alpha+2)/p}|h(z)|.
\]
\end{lemma}

\begin{proof}
Let $\varphi_z$ be the disk automorphism interchanging $0$ and $z$.  The weighted automorphism
\[
 U_zh(w)
 =
 h(\varphi_z(w))
 \left(\frac{1-|z|^2}{(1-\overline zw)^2}\right)^{(\alpha+2)/p}
\]
is an isometry of $A_\alpha^p$.  Subharmonicity and the normalization of $\mathrm dA_\alpha$ give
\[
 |U_zh(0)|\le\norm{U_zh}_{A_\alpha^p}.
\]
Since
\[
 |U_zh(0)|=(1-|z|^2)^{(\alpha+2)/p}|h(z)|,
\]
the result follows.
\end{proof}

\begin{lemma}
\label{lem:explicit-input}
Let $0<b<a<1$, $p>1/(a-b)$, and
\[
 f_b(z)=(1-z^2)^{-b}.
\]
For $\alpha=ap-2$,
\[
 \norm{f_b}_{A_\alpha^p}^p
 \le\frac{ap-1}{(a-b)p-1}.
\]
\end{lemma}

\begin{proof}
Because $|1-z^2|\ge1-|z|^2$,
\[
 |f_b(z)|^p(1-|z|^2)^{ap-2}
 \le(1-|z|^2)^{(a-b)p-2}.
\]
With $s=1-|z|^2$ and normalized area measure,
\[
 (ap-1)\int_0^1s^{(a-b)p-2}\,\mathrm ds
 =\frac{ap-1}{(a-b)p-1}.
\]
\end{proof}

At a real point $r\in(0,1)$, the two lemmas imply
\begin{equation}
 \norm{\Hcal}_{A_{ap-2}^p\to A_{ap-2}^p}
 \ge
 (1-r^2)^a\Hcal f_b(r)
 \left(\frac{(a-b)p-1}{ap-1}\right)^{1/p}.
 \label{eq:explicit-master}
\end{equation}

\subsection{Monotonicity in \texorpdfstring{$p$}{p}}

Fix $a>b$ and put $d=a-b$.  The $p$-dependent factor in \eqref{eq:explicit-master} is
\[
 G(p)=\left(\frac{dp-1}{ap-1}\right)^{1/p},
 \qquad p>\frac1d.
\]
Direct differentiation gives
\begin{equation}
 \frac{\mathrm d}{\mathrm dp}\log G(p)
 =
 \frac1{p^2}
 \left[
 \frac{bp}{(dp-1)(ap-1)}
 -\log\left(\frac{dp-1}{ap-1}\right)
 \right]>0.
 \label{eq:explicit-monotonicity}
\end{equation}
Both terms in brackets are positive.  Consequently, a strict lower bound at one exponent propagates to every larger real exponent without changing the test function.

\subsection{The certified endpoint}

Take
\[
 a=a_0=\frac{800001}{1000000},
 \qquad b=\frac45,
 \qquad d=\frac1{1000000},
 \qquad r=\frac{999}{1000},
\]
and write
\[
 \delta=1-r^2=\frac{1999}{1000000}.
\]
At $p_0=1100000$,
\[
 \frac{dp_0-1}{ap_0-1}=\frac1{8800001}.
\]

Rationalizing the denominator by
\[
 \frac{1}{1-rt}=\frac{1+rt}{1-r^2t^2}
\]
and then putting $x=t^2$ gives
\begin{align*}
 \Hcal f_{4/5}(r)
 &=\frac12\B\left(\frac12,\frac15\right)
 {}_2F_1\left(1,\frac12;\frac7{10};r^2\right)\\
 &\quad+\frac{5r}{2}{}_2F_1\left(1,1;\frac65;r^2\right).
\end{align*}
Apply the Gauss connection formula at $1$, with $1-r^2=\delta$.  The regular connection coefficients reduce, by the gamma recurrence, to $3/8$ and $-1/4$, respectively.  After the prefactors in the preceding display are included, each of the two singular terms is
\[
 \frac12\B\left(\frac45,\frac15\right)
 r^{3/5}\delta^{-4/5}.
\]
Their sum gives the leading term in the following exact expression:
\begin{align}
 \Phi
 &:=
 \delta^{4/5}\Hcal f_{4/5}(r)
 \notag\\
 &=
 \B\left(\frac45,\frac15\right)r^{3/5}
 \notag\\
 &\quad+
 \delta^{4/5}
 \left[
 \frac3{16}\B\left(\frac12,\frac15\right)
 {}_2F_1\left(1,\frac12;\frac95;\delta\right)
 -
 \frac{5r}{8}{}_2F_1\left(1,1;\frac95;\delta\right)
 \right].
 \label{eq:explicit-Phi}
\end{align}
Thus, the right-hand side of \eqref{eq:explicit-master}, divided by the proposed beta value, is
\begin{equation}
 Q=
 \delta^{1/1000000}
 \frac{\Phi}{\B(a_0,1-a_0)}
 8800001^{-1/1100000}.
 \label{eq:explicit-Q}
\end{equation}

This scalar inequality can be certified using positive rational series:
\begin{align*}
 \B\left(\frac12,\frac15\right)
 &=10\,2^{-4/5}
   \sum_{k=0}^\infty
   \frac{(4/5)_k}{k!\,2^k(5k+1)},\\
 {}_2F_1\left(1,\frac12;\frac95;\delta\right)
 &=\sum_{n=0}^\infty\frac{(1/2)_n}{(9/5)_n}\delta^n,\\
 {}_2F_1\left(1,1;\frac95;\delta\right)
 &=\sum_{n=0}^\infty\frac{n!}{(9/5)_n}\delta^n.
\end{align*}
The successive-term ratio is less than $1/2$ in the first series and less than $\delta<1/500$ in the other two.  Truncation after $81$, $9$, and $9$ terms, respectively, with geometric tail bounds and alternating Taylor bounds for the sine factors, yields
\begin{align*}
 6.26865312408603
 &<\B\left(\frac12,\frac15\right)
 <6.26865312408604,\\
 5.3454065393010
 &<\Phi
 <5.3454065393012,\\
 5.3448197717724
 &<\B(a_0,1-a_0)
 <5.3448197717726.
\end{align*}
Substitution in \eqref{eq:explicit-Q} gives the outward-rounded enclosure
\begin{equation}
 1.00008902869907
 <Q
 <1.00008902869908.
 \label{eq:explicit-certificate}
\end{equation}

\begin{proof}[Proof of Theorem \ref{thm:counterexample}]
At $p=p_0$, equations \eqref{eq:explicit-master} and \eqref{eq:explicit-certificate} give a strict excess over $\B(a_0,1-a_0)$.  Equation \eqref{eq:explicit-monotonicity} propagates the same strict inequality to every real $p\ge p_0$.  Finally, $0<a_0<1$ and $a_0p>1$, so
\[
 -1<\alpha_p=a_0p-2<p-2.
\]
This proves the theorem.
\end{proof}

\begin{remark}
The numerical check uses three positive series, geometric tail bounds, and elementary Taylor enclosures.  The companion script certifies the displayed decimals; no other part of the proof depends on it.
\end{remark}

\section{Scope of the certificate program}
\label{sec:scope}

The counterexample changes the global interpretation of the Boyd--Schur program but does not diminish its local force.  Theorems \ref{thm:p2}, \ref{thm:uniform}, and \ref{thm:small-even} identify regions in which the boundary beta profile is indeed extremal.  At other individual parameter pairs, Theorem \ref{thm:moving} remains a precise sufficient criterion.  What is impossible is a universal family of beta certificates covering every even exponent and every admissible weight.

\subsection{The unresolved case \texorpdfstring{$p=12$}{p=12}}
The first exponent beyond the full-range theorems is $m=6$.  A shifted fixed certificate $h_j=(4)_j/j!$ leads to the positive kernel
\[
 P_n(y)={}_2F_1(-n,2;6;y)
 =20\int_0^1r(1-r)^3(1-yr)^n\,\mathrm dr.
\]
The estimate
\[
 P_n(y)\le \frac{20}{y^2(n+4)(n+5)}
\]
reduces part of the interval $6/11<a<10/11$ to an explicit base moment, but the artificial singularity at $y=0$ prevents it from reaching $10/11\le a<1$.  The exact positive kernel or a quadratic moving certificate
\[
 h(z)=(1-z)^{-6}\left(1+\lambda_1(1-z)+\lambda_2(1-z)^2\right)
\]
may still decide the missing $p=12$ region.  The high-power counterexample alone gives no negative conclusion for this fixed exponent.

\subsection{A local moment criterion}
For a proposed multi-layer certificate $h_P$, the residual rows often have the Hausdorff-moment form
\[
 \Delta_j(P;c)=\int_0^1x^j\,\mathrm d\sigma_{P,c}(x),
\]
where $\sigma_{P,c}$ is a signed measure depending on the Boyd extremizer.  If $P$ makes the first $R$ residual rows vanish, a sufficient condition for all remaining rows to be nonnegative is
\[
 (-1)^R\int_{[0,t]}
 \frac{(t-x)^{R-1}}{(R-1)!}\,
 \mathrm d\sigma_{P,c}(x)\ge0,
 \qquad 0\le t\le1.
\]
This is the higher-order analogue of the one-sign-change argument used in the fixed cases.  It remains useful as a local positivity test, but Theorem \ref{thm:counterexample} implies that it must fail for at least one finite-section extremizer at $a=a_0$ for every even $p\ge1100000$.

\appendix

\section{Hypergeometric and beta identities}
\label{app:identities}

We use several standard formulas.  The beta--gamma identity and Euler's reflection formula are \cite[Eqs.~(5.12.1) and (5.5.3)]{DLMF}
\[
 \B(x,y)=\frac{\Gamma(x)\Gamma(y)}{\Gamma(x+y)},
 \qquad
 \B(a,1-a)=\frac{\pi}{\sin(\pi a)}
\]
for $x,y>0$ and $0<a<1$.  We also use
\[
 \B(a,1-a)=\int_0^\infty\frac{u^{a-1}}{1+u}\,\mathrm du.
\]
The hypergeometric integral representation and Euler transformation are, respectively, \cite[Eqs.~(15.6.1) and (15.8.1)]{DLMF}
\[
 {}_2F_1(A,B;C;z)
 =\frac{\Gamma(C)}{\Gamma(B)\Gamma(C-B)}
 \int_0^1t^{B-1}(1-t)^{C-B-1}(1-zt)^{-A}\,\mathrm dt,
\]
valid for $C>B>0$, and
\[
 {}_2F_1(A,B;C;z)
 =(1-z)^{C-A-B}{}_2F_1(C-A,C-B;C;z).
\]

We shall also use the elementary Bernstein criterion.  If
\[
 P(x)=\sum_{k=0}^{d}b_k\binom dkx^k(1-x)^{d-k},
\]
then $b_k\ge0$ for all $k$ implies $P\ge0$ on $[0,1]$; strict positivity of all coefficients gives $P>0$ on $(0,1)$.  This follows immediately from the nonnegativity of the Bernstein basis.

\section{Base moment certificates for \texorpdfstring{$m=2,3,4,5$}{m=2,3,4,5}}
\label{app:hard-small}

We verify \eqref{eq:base-moment} for $m=2,3,4,5$ in the range $m/(2m-1)<a<1$.  Fubini's theorem, beta recurrence, and Euler's reflection formula reduce each case to an elementary inequality.  The final step uses either monotonicity or the Bernstein criterion recalled in Appendix \ref{app:identities}.

\subsection{The case \texorpdfstring{$m=2$}{m=2}}

Here the hard range is $2/3<a<1$.  Put $x=1-a$, so $0<x<1/3$.  The base moment reduces to
\[
  \frac{\B(a+1,1-a)}{4a-1}\ge \B(3-3a,3a-1).
\]
In terms of $x$, this is equivalent to
\[
  (3-4x)(1-3x)\sin(\pi x)
  \le (1-x)\sin(3\pi x).
\]
Dividing by $\sin(\pi x)>0$ and using
\[
  \frac{\sin(3\pi x)}{\sin(\pi x)}=1+2\cos(2\pi x),
\]
it is enough to prove
\[
  \Phi(x)=2\cos(2\pi x)+9x-2\ge0,
  \qquad 0<x<1/3.
\]
One has $\Phi(0)=\Phi(1/3)=0$,
\[
  \Phi'(x)=9-4\pi\sin(2\pi x),
  \qquad
  \Phi''(x)=-8\pi^2\cos(2\pi x).
\]
Thus, $\Phi'$ decreases on $(0,1/4)$ and increases on $(1/4,1/3)$, with $\Phi'(0)>0$, $\Phi'(1/4)<0$, and $\Phi'(1/3)<0$.  Therefore, $\Phi$ rises once and then decreases to $0$, hence $\Phi\ge0$.  This proves the base moment for $m=2$.

\subsection{The case \texorpdfstring{$m=3$}{m=3}}

The hard range is $3/5<a<1$.  Put $x=1-a$, $0<x<2/5$.  After Fubini and the chord bound for
\[
  P_b(x)=\int_0^x\frac{v^b}{1+v}\,\mathrm{d}v,
\]
the base moment reduces to the beta comparison
\begin{align*}
 \frac{2\B(a+2,1-a)}{(6a-2)(6a-1)}
 &\ge
 \B(5a-2,6-5a)\\
 &\quad+\frac{1}{6a-2}\B(5a-1,5-5a).
\end{align*}
Using beta recurrence and reflection, this is equivalent to
\[
 \frac{\lvert\sin(5\pi x)\rvert}{\sin(\pi x)}
 \ge
 \frac{(5-6x)(2-5x)\lvert1-5x\rvert(1+5x-10x^2)}{2(2-x)(1-x)}.
\]
The elementary bound
\[
 \frac{\lvert\sin(5\pi x)\rvert}{\sin(\pi x)}
 \ge \frac{5}{2}\lvert1-5x\rvert(2-5x),
\]
for $0<x<2/5$, together with
\[
 \frac{(5-6x)(1+5x-10x^2)}{2(2-x)(1-x)}\le \frac{5}{2},
\]
proves the desired inequality.  The latter inequality is equivalent to
\[
  5-34x+85x^2-60x^3\ge0,
\]
and this follows from
\[
  85x^2-60x^3\ge61x^2,
  \qquad
  5-34x+61x^2>0.
\]

\subsection{The case \texorpdfstring{$m=4$}{m=4}}

The hard range is $4/7<a<1$.  Put $u=7(1-a)$.  The beta comparison obtained from the base moment is equivalent to
\begin{equation}
\label{eq:T4-app}
 \frac{\lvert\sin(\pi u)\rvert}{\sin(\pi u/7)}
 \ge
 \Phi_4(u)\lvert(1-u)(2-u)(3-u)\rvert,
 \qquad 0<u<3,
\end{equation}
where
\[
 \Phi_4(u)=
 \frac{(21-4u)(49-8u)(7+7u-2u^2)}{3(21-u)(14-u)(7-u)}.
\]
Since $\sin(\pi u/7)\le\pi u/7$, it suffices to prove
\[
 \lvert\sin(\pi u)\rvert
 \ge \frac{\pi u}{7}\Phi_4(u)\lvert(1-u)(2-u)(3-u)\rvert.
\]
We use
\begin{equation}
\label{eq:sine-basic}
  \sin(\pi t)\ge \pi t(1-t)(1+t(1-t)),
  \qquad 0\le t\le1.
\end{equation}
On each interval $(0,1)$, $(1,2)$, and $(2,3)$, set $t=u-k$.  After clearing denominators, \eqref{eq:sine-basic} reduces the desired inequality to the positivity of Bernstein polynomials on $[0,1]$.  The Bernstein coefficients are as follows.

For $0<u<1$, the relevant polynomial $P_0$ of degree $5$ has Bernstein coefficients
\[
39984,\ \frac{188629}{5},\ \frac{334579}{10},\ \frac{140213}{5},\ 22092,\ 16032.
\]
For $1<u<2$, with $v=u-1$, the degree $6$ polynomial has Bernstein coefficients
\[
16032,\ 18998,\ \frac{301594}{15},\ \frac{396561}{20},\ \frac{92067}{5},\ \frac{48064}{3},\ 12786.
\]
For $2<u<3$, with $v=u-2$, the degree $6$ polynomial has Bernstein coefficients
\[
12786,\ \frac{41438}{3},\ \frac{65563}{5},\ \frac{225729}{20},\ \frac{26452}{3},\ 6033,\ 3132.
\]
All are positive, and hence \eqref{eq:T4-app} holds.

\subsection{The case \texorpdfstring{$m=5$}{m=5}}

The hard range is $5/9<a<1$.  Here the one-step chord bound is no longer sharp enough.  A second-order Bernstein majorant is used.  The base moment reduces to
\begin{align*}
&\B(a+4,1-a)\B(5,10a-5)\\
&\quad\ge \frac{1}{10a-5}\Bigg[
\B(9a-4,10-9a)
+\frac{1}{10a-4}\B(9a-3,10-9a)\\
&\hspace{3.0cm}
+\frac{2}{(10a-4)(10a-3)}\B(9a-2,9-9a)
\Bigg].
\end{align*}
With $u=9(1-a)$, this becomes
\begin{equation}
\label{eq:T5-app}
 \frac{\lvert\sin(\pi u)\rvert}{\sin(\pi u/9)}
 \ge \Psi_5(u),
 \qquad 0<u<4,
\end{equation}
where
\[
\Psi_5(u)=
\frac{\lvert(u-1)(u-2)(u-3)(u-4)\rvert(36-5u)(81-10u)
(46u^3-525u^2+1431u+324)}{24(36-u)(27-u)(18-u)(9-u)}.
\]
Using \eqref{eq:sine-basic} and $\sin(\pi u/9)\le\pi u/9$, the inequality is reduced on each interval $(k,k+1)$, $k=0,1,2,3$, to positivity of a Bernstein polynomial of degree $8$ on $[0,1]$.  The coefficient lists are:
\[
\begin{array}{ll}
 k=0:& 11337408,\ 5900526,\ \frac{31970052}{7},\ \frac{151276977}{28},\ \frac{71019729}{10},\\
& \frac{246828201}{28},\ \frac{280537807}{28},\ \frac{41678983}{4},\ 9881304;\\[2mm]
 k=1:& 9881304,\ \frac{47252753}{4},\ \frac{369718127}{28},\ \frac{98935164}{7},\\
& \frac{1022151211}{70},\ \frac{204135209}{14},\ 14039598,\ 12959132,\ 11339024;\\[2mm]
 k=2:& 11339024,\ 12553672,\ \frac{91789826}{7},\ \frac{184137459}{14},\\
& \frac{25548489}{2},\ \frac{337273767}{28},\ \frac{308252871}{28},\ 9691083,\ 8109396;\\[2mm]
 k=3:& 8109396,\ 8555058,\ \frac{235547271}{28},\ \frac{109522341}{14},\\
& \frac{483925311}{70},\ \frac{23152911}{4},\ \frac{31700180}{7},\ 3190286,\ 1807872.
\end{array}.
\]
Every coefficient is positive.  Therefore, \eqref{eq:T5-app} holds and the base moment for $m=5$ follows.

\section{A shifted certificate for the next case \texorpdfstring{$m=6$}{m=6}}
\label{app:m6}

For $p=12$, the present calculation stops at a shifted certificate.  We include the reduction because it indicates where a moving certificate may be needed.

Let $m=6$ and take
\[
  h_j=\frac{(4)_j}{j!}.
\]
Then
\[
(B_yh)_j=y^j(1-y)^{-4}
\sum_{n=0}^{\infty}
\frac{(j+6)_n(6)_n}{(12a)_n n!}(1-y)^nP_n(y),
\]
where
\[
  P_n(y)={}_2F_1(-n,2;6;y)=20\int_0^1r(1-r)^3(1-yr)^n\,\mathrm{d}r.
\]
The bound
\[
  P_n(y)\le \frac{20}{y^2(n+4)(n+5)}
\]
gives the sufficient reduction, valid for $a<10/11$, to the base moment
\[
\B(a+3,1-a)\B(4,12a-4)
\ge
\int_0^\infty x^{-a}(1+x)^{-4}
\left[\int_0^x\frac{v^{12a-5}}{(1+v)^3}\,\mathrm{d}v\right]\,\mathrm{d}x.
\]
This reduction is too crude near $a=1$ because of the artificial singularity produced by $y^{-2}$.  Deciding whether the beta formula remains valid throughout the $p=12$ range requires the exact positive kernel $P_n(y)$ or a different device, such as the moving quadratic certificate
\[
  h(z)=(1-z)^{-6}(1+\lambda_1(1-z)+\lambda_2(1-z)^2).
\]

\section*{Funding}
The work of Hasi Wulan was supported in part by the National Natural Science Foundation of China (Grant Nos. 12271328 and 12371131).

The work of Mengmeng Zhou was supported in part by the National Natural Science Foundation of China (Grant No. 12371131), the STU Scientific Research Initiation Grant (No. NTF23004), the LKSF STU--GTIIT Joint Research Grant (No. 2024 LKSFG06), and the Guangdong Basic and Applied Basic Research Foundation (Grant No. 2023A1515010614).

The work of Jian-Feng Zhu was supported by the National Natural Science Foundation of China (Grant No. 12271189), the Natural Science Foundation of Guangdong Province (Grant Nos. 2024A1515010467 and 2026A1515012333), the STU Scientific Research Initiation Grant (No. NTF25017T), and the Fujian Alliance of Mathematics (Grant No. 2023SXLMMS07).

\section*{Data availability}
No datasets were generated or analyzed during the current study.

\section*{Conflicts of interest}
The authors declare that they have no financial or non-financial conflicts of interest that are directly or indirectly related to the work submitted for publication.

\section*{AI declaration}
The counterexample in Theorem \ref{thm:counterexample}, based on the test function
\[
 f_0(z)=(1-z^2)^{-4/5}
 =\sum_{k=0}^{\infty}\frac{(4/5)_k}{k!}z^{2k},
\]
was constructed with the assistance of OpenAI Codex.  The authors independently verified the argument and take full responsibility for the final manuscript.

\end{document}